\documentclass[draft]{amsart}

\usepackage{pgf,tikz}
\usetikzlibrary{arrows}
\usepackage{cite}
\usepackage{stmaryrd}

\usepackage{amssymb,graphicx,enumerate,color}
\usepackage{multirow}

\usepackage{graphicx}

\renewcommand{\l}{\lambda}

\newcommand{\beqs}{\begin{equation*}}
\newcommand{\eeqs}{\end{equation*}}
\numberwithin{equation}{section}
 \theoremstyle{plain}
\newtheorem{theorem}{Theorem}[section]
\newtheorem{lemma}[theorem]{Lemma}

\newtheorem{conjecture}[theorem]{Conjecture}

\theoremstyle{remark}

\theoremstyle{empty}

\begin{document}

\makeatletter
\def\imod#1{\allowbreak\mkern10mu({\operator@font mod}\,\,#1)}
\makeatother

\author{Alexander Berkovich}
   \address{Department of Mathematics, University of Florida, 358 Little Hall, Gainesville FL 32611, USA}
   \email{alexb@ufl.edu}

\author{Ali Kemal Uncu}
   \address{Department of Mathematics, University of Florida, 358 Little Hall, Gainesville FL 32611, USA}
   \email{akuncu@ufl.edu}

\title[On some polynomials and series of Bloch--P\'olya type]{On some polynomials and series of Bloch--P\'olya type}



\begin{abstract} We will show that $(1-q)(1-q^2)\dots (1-q^m)$ is a polynomial in $q$ with coefficients from $\{-1,0,1\}$ iff $m=1,\ 2,\ 3,$ or 5 and explore some interesting consequences of this result. We find explicit formulas for the $q$-series coefficients of $(1-q^2)(1-q^3)(1-q^4)(1-q^5)\dots$ and $(1-q^3)(1-q^4)(1-q^5)(1-q^6)\dots$. In doing so, we extend certain observations made by Sudler in 1964. We also discuss the classification of the products $(1-q)(1-q^2)\dots (1-q^m)$ and some related series with respect to their absolute largest coefficients.
\end{abstract}   
   


\date{\today}
   
\maketitle

\section{Introduction and Background}

Polynomials with coefficients from the set $\{-1,0,1\}$ have been first studied by Bloch and P\'olya \cite{BlochPolya}. Their study sparked interest especially about the roots of these polynomials. An interested reader may refer to these prominent results (listed in order of publication) \cite{Littlewood}, \cite{BorweinErdeyliKos}, \cite{Erdelyi}, and \cite{BorweinMossinghoff}. The first two references focus on \textit{Littlewood polynomials}, i.e. polynomials where all coefficients are $\pm 1$. In the same spirit, we would like to call polynomials (and series) with integer coefficients from the set $\{-1,0,1\}$, \textit{Bloch--P\'olya type} polynomials (and series, resp.).

We start by defining a \textit{$q$-Pochhammer symbol} or a \textit{rising $q$-factorial}. For variables $a,\ q$ and a non-negative integer $L$, we define
\begin{align*}
(a;q)_L &= \prod_{i=0}^{L-1} (1-aq^i),\\
(a;q)_\infty &= \lim_{L\rightarrow\infty} (a;q)_L, \text{ for }|q|<1.
\end{align*}

The rising $q$-factorials have been studied extensively, one notable example is Euler's Pentagonal Number Theorem \cite[Cor 1.7, pg 11]{Theory_of_Partitions}.

\begin{theorem}[Euler's Pentagonal Number Theorem, 1750]\label{EPNT_THM} \begin{align}
\label{EPNT} (q&;q)_\infty = \sum_{n=-\infty} ^{\infty} (-1)^n q^{n(3n-1)/2} \\\nonumber &= 1 -  q-  q^2+ q^5+  q^7-  q^{12}-  q^{15}+q^{22}+  q^{26}- q^{35}- q^{40}+  q^{51}+q^{57}-q^{70}-q^{77}+q^{92}\dots .
\end{align}
\end{theorem}

The identity \eqref{EPNT} shows that the $q$-Pochhammer symbol $(q;q)_\infty$ can be represented as Bloch--P\'olya type series. 

We will also require the $q$-binomial theorem \cite[Thm 2.1, p. 17]{Theory_of_Partitions}:

\begin{theorem}[$q$-Binomial Theorem]\label{qBin}
\begin{equation}\label{qbinomial}
\sum_{n\geq 0} \frac{(a;q)_n}{(q;q)_n} t^n = \frac{(ta;q)_\infty}{(t;q)_\infty.}
\end{equation}
\end{theorem}

Next, we define the following family of sums,
\begin{align} \label{recurr_F_finite}F_{k,M}(q)&:= \sum_{j=0}^{M} q^{kj}(q;q)_j,\\
\label{recurr_F_inf} F_{k}(q)&:=\lim_{M\rightarrow\infty}F_{k,M}(q)= \sum_{j\geq0} q^{kj}(q;q)_j,
\end{align} for positive integers $k$ and $M$. The series $F_{k}(q)$ can be made convergent by picking $|q|<1$.


These series were introduced by Eden and they are closely related with the theory of partitions. Eden observed \[(-q)^k F_{k}(q) = \sum_{n,m > 0} (-1)^m p_{k}(n,m) q^n,\] where $p_{k}(n,m)$ is the number of non-empty partitions of $n$ into exactly $m$ parts where the largest part appears $k$ times and all the other parts appear distinctly \cite{Eden}. Moreover, the series $F_1(q)$ plays an instrumental role in Euler's original proof of Theorem~\ref{EPNT_THM} \cite{AndrewsEuler}. Recently $F_{i,M}(q)$ for $i=1,\ 2,$ and $3$ arose naturally in our studies of partitions with bounded gaps between largest and smallest parts \cite{BerkovichUncu7}.

In the following sections, among other observations, we will prove the next two theorems.

\begin{theorem}\label{qq_m_not_Bloch_Polya} For $m\in\mathbb{Z}_{\geq0}$, $(q;q)_m$ is of Bloch--P\'olya type iff $m=0,\, 1,\, 2,\, 3$ or 5.
\end{theorem}

\begin{theorem}\label{highlight_F} For $k\geq 7$, there is no polynomial $f(q)$ such that $F_k(q) + f(q)$ is of Bloch-P\'olya type.
\end{theorem}

C. Sudler, in his 1964 papers \cite{Sudler1, Sudler2} studied the maximum coefficient, $a_m^*$, of the power series expansion of $(q;q)_m$. He noted that $a_{m}^*$ is unbounded as $m$ gets larger using a special case of Theorem~\ref{qBin} and that $\log\, a_{m}^*\sim K m$ where $K>0$ using Cauchy's integral formula, in the respective papers. The unboundedness of $a_m^*$ was shown by observing that $(q^3;q)_\infty$ has unbounded $q$-series coefficients.

In Section~\ref{Section_Pochhammer}, we start with some observations about pentagonal numbers. After that we give explicit formulas for the $q^j$'s coefficient in the power series of both $(q^2;q)_\infty$ and $(q^3;q)_\infty$, for any $j\in \mathbb{Z}_{\geq 0}$. This section is finalized with a proof of Theorem~\ref{qq_m_not_Bloch_Polya}. Section~\ref{Section_F} starts by proving a recurrence relation for $F_{k,M}(q)$ polynomials, and the rest of the section deals with Bloch--P\'olya properties of $F_k(q)$ series. We discuss the classification of the polynomials $(q;q)_m$, and the series $F_k(q)$ with respect to their coefficients in a broader perspective than the Bloch--P\'olya property in Section~\ref{Beyond}.

\section{$q$-series Coefficients of $(q;q)_m$, $(q^2;q)_\infty$, and $(q^3;q)_\infty$. Proof of Theorem~\ref{qq_m_not_Bloch_Polya}}\label{Section_Pochhammer}

We start by observing that the minimum gap between pentagonal numbers increase. An alternative way of writing Theorem~\ref{EPNT_THM} is \[(q;q)_\infty = 1+\sum_{n=1}^\infty (-1)^n( q^{n(3n-1)/2} + q^{n(3n+1)/2}).\] Let \begin{equation}\label{penta_families}p_1(n):= \frac{n(3n-1)}{2}\ \text{  and  }\ p_2(n) := \frac{n(3n+1)}{2}\end{equation} be the two families of pentagonal numbers for $n\geq 0$. Observe that there is a natural order between these families \begin{equation}\label{penta_ordering}0=p_1(0)=p_2(0)<1=p_1(1) < 2=p_2(1) < 5=p_1(2)<\dots < p_1(n)<p_2(n)<p_1(n+1)<\dots\end{equation} for any $n\geq 1$.
Also note that \begin{equation}\label{difference1} p_2(n)-p_1(n) = n\ \text{ and }\ p_1(n+1)-p_2(n) = 2n+1.\end{equation} This proves that 

\begin{lemma}\label{penta_gap} For any $M>0$ the gap between successive pentagonal numbers  is
\begin{equation*} p_2(n)-p_1(n) >M\ \text{ and }\ p_1(n+1)-p_2(n) >M,\end{equation*} for all $n>M$.
\end{lemma}

We will use Lemma~\ref{penta_gap} to find a suitable separation point for the tail of the series \eqref{EPNT} in the following theorems.

\begin{theorem}\label{q2;q_inf_thm} The power series of \[(q^2;q)_\infty := \sum_{j\geq 0} a_j q^j\] is of Bloch--P\'olya type. Furthermore, for any $j\in\mathbb{Z}_{\geq0}$ there exists unique $n\in\mathbb{Z}_{\geq0}$ such that \[p_2(2n)\leq j< p_2(2n+2)\] and
\begin{equation}\label{aj_coeffs}a_j = \left\{ \begin{array}{cl}
1, & \text{if }\ p_2(2n)\leq j < p_1(2n+1) ,\\
-1, &\text{if }\ p_2(2n+1)\leq j < p_1(2n+2),\\
0, & \text{otherwise.}
\end{array} \right.\end{equation}
\end{theorem}

\begin{proof} We start with \begin{equation}\label{q2q_inft}(q^2;q)_\infty = \frac{(q;q)_\infty}{1-q} = (1+q+q^2+q^3+\dots)(1 -  q-  q^2+ q^5+  q^7-  q^{12}-  q^{15}+\dots),\end{equation} which is clear by geometric series and Theorem~\ref{EPNT_THM}. Openly evaluating the latter product is enough to demonstrate this result:
\begin{align*}
(q^2&;q)_\infty =(1+q+q^2+q^3+q^4+q^5+\dots)(1 -  q-  q^2+ q^5+  q^7-  q^{12}-  q^{15}+q^{22}+  q^{26}- q^{35}- 
\dots)\\
 &=1+q+q^2+q^3+q^4 + q^5 + q^6 + q^7 + q^8 + q^9 +q^{10}+ q^{11} +q^{12}+q^{13}+ q^{14} +q^{15}+q^{16} +q^{17}\dots\\
  &\hspace{.65cm}-q-q^2-q^3-q^4 - q^5 - q^6 - q^7 - q^8- q^9 -q^{10}- q^{11} -q^{12}-q^{13}- q^{14} -q^{15}-q^{16} -q^{17}\dots\\
  &\hspace{1.25cm}-q^2 -q^3-q^4 - q^5 - q^6 - q^7 - q^8 - q^9 -q^{10}- q^{11} -q^{12}-q^{13}- q^{14} -q^{15}-q^{16} -q^{17}\dots\\
  &\hspace{3.55cm}+ q^5 + q^6 + q^7 + q^8+ q^9 +q^{10}+ q^{11} +q^{12}+q^{13}+ q^{14} +q^{15}+q^{16} +q^{17}\dots\\
  &\hspace{5.1cm} + q^7 + q^8 + q^9 +q^{10}+ q^{11} +q^{12}+q^{13}+ q^{14} +q^{15}+q^{16} +q^{17}\dots\\
  &\hspace{9.1cm} -q^{12}-q^{13}- q^{14} -q^{15}-q^{16}-q^{17}\dots\\
  &\hspace{11.8cm}-q^{15}-q^{16}-q^{17}\dots\\[-1.5ex]\\
  &=1\hspace{.65cm}-q^2 -q^3 -q^4\hspace{1.5cm}+q^7 +q^8+q^9+q^{10}+ q^{11}\hspace{2.7cm} -q^{15}-q^{16}-q^{17}\dots
\end{align*}

The sign changes at every other non-zero coefficient term in the pentagonal numbers series \eqref{EPNT} makes sure that the coefficients of \eqref{q2q_inft} are in the set $\{-1,0,1\}$, when $(q;q)_\infty$ is divided by $(1-q)$. 

With the 0 coefficients explicitly written, we have \begin{equation}\label{explicit_formula} (q^2;q)_\infty = 1+0\,q-q^2 -q^3 -q^4+0\,q^5+0\,q^6+q^7 +q^8+q^9+q^{10}+ q^{11}+0\,q^{12}+0\,q^{13}+0\,q^{14}-q^{15}\dots.\end{equation}

For $n\geq 1$, it is clear that the $n$-th non-zero coefficient block starts at $p_2(n-1)$. This block is of the size $2n-1$, and its coefficients are all $(-1)^{n+1}$. Moreover, a zero coefficient block of size $n$ follows the $n$-th non-zero coefficient block.
\end{proof}

Next we observe that, for $n\geq 0$ and $r=0,$ $1$, \begin{equation}
\sum_{2j+r\in \mathcal{I}_n} a_{2j+r} = -1,
\end{equation}
where $\mathcal{I}_n := \{p_2(2n),\dots,p_1(2n+2)-1\}$ and $a_j$ is as in Theorem~\ref{q2;q_inf_thm}. It is clear that the number of $-1$ coefficients $a_j$ in $p_2(2n)$ to $p_1(2n+2)-1$ is $2$ more than the $+1$ coefficients. This observation is true due to the number of zero coefficients in this interval being an odd number. Another way of seeing this is to observe $p_2(2n)$ and $p_2(2n+1)$ having the same parity, for $n\geq 0$.



Let \begin{equation}\label{b_i_coeffs}(q^3;q)_\infty = \sum_{i\geq 0} b_i q^i\end{equation} be the power series representation. Since $(q^3;q)_\infty = (q^2;q)_\infty / (1-q^2)$, it is clear that \[b_i = \sum_{\substack{j\geq 0\\ i-2j\geq0}} a_{i-2j},\] where $a_j$ is as in Theorem~\ref{q2;q_inf_thm}. With that, one can prove the following.

\begin{theorem}\label{q3;q_infty} For every integer $i\geq 0$, let $b_i$ be as in \eqref{b_i_coeffs}, there exists a unique integer $n\geq 0$ such that \[p_1(2n)-2\leq i \leq p_1(2n+2)-3.\] Then, the power series coefficients $b_i$ of $(q^3;q)_\infty$ is given explicitly by
\[b_i = \left\{ \begin{array}{cl}
-n, & \text{if  }\ p_1(2n)-2 \leq i \leq p_2(2n)-1,\\[-1.5ex]\\
1-n + \lfloor \frac{i-p_2(2n)}{2} \rfloor, & \text{if  }\ p_2(2n) \leq i \leq p_1(2n+1)-2,\\[-1.5ex]\\
1+n, & \text{if  }\ p_1(2n+1)-1 \leq i \leq p_2(2n+1)-2\text{ and }i\equiv p_2(2n)\ (\text{mod }2),\\[-1.5ex]\\
n, & \text{if  }\ p_1(2n+1)-1 \leq i \leq p_2(2n+1)-2\text{ and }i\not\equiv p_2(2n)\ (\text{mod }2),\\[-1.5ex]\\
n - \lceil \frac{i-p_2(2n+1)}{2}\rceil, & \text{if  }\ p_2(2n+1)-1 \leq i \leq p_1(2n+2)-3,\\
\end{array}\right.\]
where $\lfloor x \rfloor$ is the greatest integer $\leq x$, and $\lceil x \rceil$ is the smallest integer $\geq x$.
\end{theorem}

As an example, if $i=10^{100}$, then $n=40824829046386301636621401245098189866099124677611$. Moreover, for this particular $i$ the second case of the formula above applies. Hence, after the addition of three numbers we get, $b_{i} = -19888090251390639910818356938628130689602741018379$. 

Theorem~\ref{q3;q_infty} already says that series expansion of $(q^3;q)_\infty$ is not of Bloch--P\'olya type. Moreover, \textit{every integer occurs as a coefficient of $(q^3;q)_\infty$ infinitely many times}. For illustrative purposes, some first appearances of non-zero coefficient sizes are \begin{equation}\label{coeffSizes}(q^3;q)_\infty = 1 +\dots + 2q^{11}+\dots +3q^{34} +\dots+4q^{69}+\dots + 5q^{116} + \dots + 6q^{175}+\dots + 7q^{246}+\dots .\end{equation}

We remark that Euler's Pentagonal Number Theorem (Theorem~\ref{EPNT_THM}) has a partition-theoretic interpretation. A finite sequence $\pi=(\lambda_1,\lambda_2,\dots)$ of non-increasing positive integers is called a \textit{partition}. Let $\mathcal{D}$ be the set of partitions into distinct parts (i.e. $\lambda_i\not=\lambda_j$ for $i\not=j)$, $\nu(\pi)$ be the number of parts of partition $\pi$, and $|\pi|$ is the sum of all the parts of $\pi$. The empty sequence is the unique partition of zero. Then \eqref{EPNT} can be interpreted as \[(q;q)_\infty = \sum_{n\geq 0} \left( \sum_{\substack{\pi \in \mathcal{D}\\ |\pi|=n} } (-1)^{\nu(\pi)} \right) q^n.\] Similar to this interpretation one can also interpret Theorem~\ref{q2;q_inf_thm} and \ref{q3;q_infty} as a partition theorem.
\begin{theorem}
\[a_j = \sum_{\substack{\pi \in \mathcal{D}\\ |\pi|=j \\ s(\pi)>1} }(-1)^{\nu(\pi)}\ \text{and}\ b_i = \sum_{\substack{\pi \in \mathcal{D}\\ |\pi|=i \\ s(\pi)>2} }(-1)^{\nu(\pi)}, \]
where $s(\pi)$ is the smallest part of partition $\pi$, and $a_j$ and $b_i$ are defined as in Theorem~\ref{q2;q_inf_thm} and \ref{q3;q_infty}. 
\end{theorem}

Now we can give an easy proof of Theorem~\ref{qq_m_not_Bloch_Polya}. To this end, we will define $\llfloor q^i \rrfloor f(q)$ to be the $q^i$'s coefficient in the power series expansion of $f(q)$.

\begin{proof} \textbf{(Theorem~\ref{qq_m_not_Bloch_Polya})} Initial cases of $q$-factorials can easily be checked to be Bloch--P\'olya type polynomials for $m=0$ to $5$, except for $m=4$: 
\begin{align}
\nonumber(q;q)_0 &=1,\\
\nonumber(q;q)_1 &= 1-q,\\
\nonumber(q;q)_2 &= 1-q-q^2+q^3,\\
\nonumber(q;q)_3 &= 1-q-q^2+q^4+q^5-q^6,\\
\nonumber(q;q)_4 &= 1-q-q^2+2\, q^5-q^8-q^9+q^{10},\\
\nonumber(q;q)_5 &= 1-q-q^2+q^5+q^6+q^7-q^8-q^9-q^{10}+q^{13}+q^{14}-q^{15}.\\
\intertext{It is clear that $\llfloor q^5 \rrfloor (q;q)_4 = 2$. Also observe that }
\nonumber\llfloor q^7 \rrfloor (q;q)_6 &=2,\\
\nonumber\llfloor q^{12}\rrfloor(q;q)_m &= -2,\text{ for }\ 7\leq m\leq 9,\\
\nonumber\llfloor q^{15} \rrfloor(q;q)_{10} &= -2,\\
\nonumber-3\leq \llfloor q^{2m+22} \rrfloor (q;q)_m &\leq -2,\text{ for }\ 11\leq m\leq 20.
\intertext{Another example that will come in handy is}
\label{specialcase}\llfloor q^{51} \rrfloor (q;q)_{41}&=2.
\end{align}

For some $m\geq 1$, choose $(a,q,t)= (0,q,q^{m})$ in the $q$-binomial theorem \eqref{qbinomial}. Multiplying both sides of this special case with $(q;q)_\infty$, one can easily show that \begin{equation}\label{q_sub_m}
(q;q)_{m-1} = \sum_{i\geq 0} q^{mi} (q^{i+1};q)_\infty= (q;q)_\infty + q^{m}(q^2;q)_\infty + q^{2m}(q^3;q)_\infty \dots.
\end{equation}

Note that for any $2m \leq n <3m$ the contribution for the coefficient of $q^n$ of $(q;q)_{m-1}$ comes only from the first three terms of the expansion in \eqref{q_sub_m}. Keeping \eqref{coeffSizes} in mind, for $m > 69$ one can deduce that \[2\leq \llfloor q^{2m+69} \rrfloor (q;q)_{m-1}\leq 6,\] since the coefficients of $q^{2m+69}$ and $q^{m+69}$ of $(q;q)_\infty$ and $(q^2;q)_\infty$ are in the set $\{-1,0,1\}$ by Theorem~\ref{EPNT_THM} and Theorem~\ref{q2;q_inf_thm}, respectively. We can now directly verify the claim for the intermediate interval of unchecked $m$ values ($m\geq 22$ and the inequality $2m \leq 2m+69 <3m$ does not hold) that \[2\leq\llfloor q^{2m+69} \rrfloor (q;q)_{m-1}\leq12\ \text{for}\ 22\leq m \leq69\ \text{and}\ m\not= 42.\] Recall that $m=42$ case was handled in \eqref{specialcase} above.
\end{proof}

\section{Recurrence Relations for $F_{k,M}(q)$ and a Proof of Theorem~\ref{highlight_F}}\label{Section_F}

We start by the recurrence relations for the $F_{k,M}(q)$ functions.

\begin{lemma}\label{finite_recurrence_lemma}For $k \geq 1$  \[q^{k+1} F_{k+1,M}(q) = 1+(q^k-1)F_{k,M}(q) - q^{k(M+1)}(q;q)_{M+1}.\]
\end{lemma}

\begin{proof} Observe that \[(q^k-1)F_{k,M}(q) = \sum_{n=1}^{M+1} q^{kn} (q;q)_{n-1} - \sum_{n=0}^{M}q^{kn}(q;q)_n = q^{k+1}F_{k+1,M-1}(q) + q^{k(M+1)}(q;q)_M -1.\]
Adding $q^{(k+1)(M+1)}(q;q)_M$, we get the term $q^{k+1}F_{k+1,M}(q)$ on the right side of the equation. Isolating this term yields the result.
\end{proof}

Let $\mathcal{D}_{1,k}$ be the set of non-empty partitions into distinct parts congruent to 1 modulo $k$, and $\nu(\pi)$ be the number of parts of the partition $\pi$. We note that \begin{equation}\label{GF_of_1_mod_k_kind} \sum_{\substack{\pi\in\mathcal{D}_{1,k}\\l(\pi) \leq kM+1}} (-1)^{\nu(\pi)+1} q^{|\pi|} = \sum_{j=0}^{M} q^{kj+1} (q;q^k)_j = 1-(q;q^k)_{M+1},\end{equation} where $l(\pi)$ is the largest part of $\pi$.

The $q$-factorial $(q;q^k)_{M+1}$ is the generating function for the partitions $\pi^*$ into distinct parts $\leq kM+1$, each 1 modulo $k$, counted with the weights $(-1)^{\nu(\pi^*)}$. The summand of the middle term of \eqref{GF_of_1_mod_k_kind} is the generating function for partitions into distinct parts, each 1 modulo $k$ counted with the weight $(-1)^{\nu(\pi)+1}$ where the largest part is $jk+1$. Summing from $j=0$ to $M$, we get the generating function for the number of partitions into distinct parts $\leq kM+1$, each $1$ modulo $k$. This justifies the first equality of \eqref{GF_of_1_mod_k_kind}. The second equality can also be clarified in the same manner. We multiply $(q;q^k)_{M+1}$ by $-1$ to match the weight $(-1)^{\nu(\pi)+1}$ and add 1 to remove the empty partition from our calculations.

We will later refer to this special case of \eqref{GF_of_1_mod_k_kind}, where $k$ is 1: \begin{equation}\label{GF_of_F_1_case}qF_{1,M}(q) =\sum_{j=0}^{M} q^{j+1}(q;q)_j =  1-(q;q)_{M+1}.\end{equation} This special case also appears in \cite[(5)]{AndrewsEuler}.



Now we can prove some results about the coefficients of $F_i(q)$ functions.

\begin{theorem}\label{main_F_THM}
\begin{enumerate}[i.]
\item \label{F1_Bloch_Polya_type_THM}$F_1(q)$ is of Bloch--P\'olya type,
\item \label{F2_Bloch_Polya} $F_2(q)$ is of Bloch--P\'olya type,
\item \label{F3_F4_Bloch_Polya_type}$F_3(q)-q^9 $ and $F_4(q)-q^{16} +q^{18} + q^{30} - q^{31}$ are both of Bloch--P\'olya type,
\item $F_5(q)$ is not Bloch--P\'olya type series and there is no polynomial $f(q)$ such that $F_5(q)+f(q)$ is one,

\item $F_6(q)-f_6(q)$ is Bloch--P\'olya type series, where \begin{align*} f_6(q):=& \, q^{29}-q^{32}+q^{36}-q^{38}+q^{43}-q^{45}+q^{50}-q^{56
}+q^{57}+q^{58}-q^{62}-q^{63}+q^{64}\\&+q^{71}-q^{80}-q^{
81}+q^{84}+q^{85}+q^{106}-q^{110}-q^{239}+q^{241}+q^{280
}-q^{281},
\end{align*}
\item and for $k\geq 7$, there is no polynomial $f(q)$ such that $F_k(q) + f(q)$ is of Bloch-P\'olya type.
\end{enumerate}
\end{theorem}

The item $vi.$ of Theorem~\ref{main_F_THM} is the earlier highlighted Theorem~\ref{highlight_F}.

\begin{proof}
\begin{enumerate}[i.]
\item Taking the limit $M\rightarrow\infty$ on the extreme sides of \eqref{GF_of_F_1_case}, combined with \eqref{EPNT}, we have  \begin{equation}\label{qF1_open_form}qF_1(q) = 1-(q;q)_\infty =  q+  q^2- q^5-  q^7+  q^{12}+  q^{15}\dots.\end{equation} This is enough to show that $F_1(q)$ is of Bloch--P\'olya type.\\

\noindent The proofs of the rest of the cases ii. and iii. will rely a combination on Theorem~\ref{EPNT_THM}, Theorem~\ref{qq_m_not_Bloch_Polya}, Lemma~\ref{penta_gap}, and Lemma~\ref{finite_recurrence_lemma} with $M\rightarrow \infty$. The combination of Lemma~\ref{finite_recurrence_lemma} with $M\rightarrow \infty$, together with \eqref{qF1_open_form} yields
\begin{equation}\label{recurr_back} q^{k(k+1)/2}F_{k}(q) = \sum_{i=0}^{k-1} (-1)^i (q^{k-i};q)_{i} q^{(k-1-i)(k-i)/2} + (-1)^k (q;q)_{k-1}(q;q)_\infty, 
\end{equation} for $k\geq 1$.

\item The difference between successive pentagonal numbers (which appear in the exponent of $q$) is greater than 1 for exponents of $q$ greater than or equal to $p_1(2)=5$, by Lemma~\ref{penta_gap}. Therefore, the series $ q^5+  q^7-  q^{12}-  q^{15}\dots $ and $q( q^5+  q^7-  q^{12}-  q^{15}\dots ) = q^6 + q^8 -q^{13} - q^{16}\dots$ does not share any common exponents of $q$. Hence, their difference $(1-q)( q^5+  q^7-  q^{12}-  q^{15}\dots )$ has all the exponents of $q$ greater than or equal to $5$ and it remains Bloch--P\'olya type series.

Using \eqref{recurr_back}, we get \[q^3F_2(q) = -1 + 2q + (1-q)(q;q)_\infty.\] Using \eqref{EPNT} once again

\begin{align*}q^3F_2(q)&=-1 + 2q + (1-q)( 1 -  q-  q^2+ q^5+  q^7-  q^{12}-  q^{15}\dots )\\
&= -1 +2q + (1-q)(1-q -q^2) + (1-q)( q^5+  q^7-  q^{12}-  q^{15}\dots )\\ 
&= q^3 +  (1-q)( q^5+  q^7-  q^{12}-  q^{15}\dots ).\end{align*}

Above line with the previous observation shows that $q^3F_2(q)$ is of Bloch--P\'olya type. One can divide the series expansions of $q^3F_2(q)$ with $q^3$ to show the claimed results.

\item For $k\geq 3$, one gets power series coefficients with modulus $\geq2$ in the expansion of $F_k(q)$, but in the initial cases there are only finitely many exceptions which can be corrected. Using \eqref{recurr_back}, we get 
\begin{align*}
q^6F_3(q) &= q^3-q(1-q^2)+(q;q)_2-(q;q)_2(q;q)_\infty,\\
 &= 1-2q-q^2+3q^3-(1-q-q^2+q^3)(q;q)_\infty,\\
q^{10}F_4(q) &= q^6-q^3(1-q^3) + q(1-q^2)(1-q^3) - (q;q)_3+ (q;q)_3(q;q)_\infty,\\
&= -1+2q+q^2-2q^3-2q^4-q^5+4q^6 + (1-q-q^2+q^4+q^5-q^6)(q;q)_\infty.
\end{align*}

The $q$-factorials $(q;q)_2$ and $(q;q)_3$ have degrees $3$ and $6$, respectively. Difference between the pentagonal numbers are larger than $3$ and $6$ starting from the pentagonal numbers $22$ and $70$ by Lemma~\ref{penta_gap}. Using Theorem~\ref{EPNT_THM} and splitting the series at these pentagonal numberswe get
\begin{align}
\label{tail1}q^6F_3(q) =& q^6 +q^9 - q^{10} + q^{12} - q^{13} - q^{14} + 2q^{15}-q^{16}-q^{17}+q^{18}\\
\nonumber& -(q;q)_2(q^{22}+q^{26}- q^{35}- q^{40}+  q^{51}\dots),\\
\label{tail2}q^{10}F_4(q) =& q^{10}+q^{14}-q^{15}+q^{18}-q^{19}-q^{20}+q^{21}+q^{22}-q^{23}-q^{24}+2q^{26}-2q^{28}\\
\nonumber&+q^{30}+q^{31}-q^{32}-q^{35}+q^{36}+q^{37}-q^{39}-2q^{40}+2q^{41}+q^{42}-{q}^{44}\\
\nonumber&-q^{45}+q^{46}+q^{51}-q^{52}-q^{53}+q^{55}+q^{56}-q^{58}-q^{59}+q^{61}+q^{62}-q^{63}\\
\nonumber&-(q;q)_3(q^{70}+q^{77}-q^{92}-q^{100}+q^{117}\dots).
\end{align}

The fact that $(q;q)_2$ and $(q;q)_3$ being of Bloch-P\'olya type, ensures that the tail ends of \eqref{tail1} and \eqref{tail2} are of Bloch-P\'olya type. The explicit equations \eqref{tail1} and \eqref{tail2} are enough to prove the claims. All one needs to do is to divide both sides of these equations with $q^6$ and $q^{10}$, respectively and add in the claimed correction factors.

\item The argument for non-zero coefficients being $\pm 1$ for the tail end can be used in the opposite direction as well. By \eqref{recurr_back} we get $q^{15} F_5(q) = \dots - (q;q)_4(q;q)_\infty$. This implies that \begin{equation}\label{tail3} F_5(q) = P(q) + (q;q)_4(q^{161}+ q^{172}  -q^{195} -q^{207}\dots), \end{equation} where $P(q)$ is a polynomial of degree 150. Since $(q;q)_4 = 1-q-q^2+2q^5-q^8-q^9+q^{10}$, the tail of \eqref{tail3} can not be of Bloch--P\'olya type. Hence, $F_5(q)$ is neither Bloch--P\'olya type series, nor it can be made to be one by adding a polynomial correction term. 

\item Similar to cases i.-iii., as $(q;q)_5$ is Bloch--P\'olya polynomial, one can conclude $F_6(q)$, subject to a polynomial correction term $f_6(q)$, can be made a Bloch--P\'olya type series. More precisely, \[F_6(q)-f_6(q) = Q(q) + (q;q)_5(q^{355}+q^{371}-q^{404}-q^{421}\dots),\] where $Q(q)$ is a Bloch--P\'olya polynomial of degree 339.
\item (Proof of Theorem~\ref{highlight_F}) By Theorem~\ref{qq_m_not_Bloch_Polya} we know that $(q;q)_{k-1}$ is not of Bloch--P\'olya type for $k\geq 7$. Hence, following the steps of case iv., the tail of $F_k(q)$ for $k\geq 7$ cannot be of Bloch--P\'olya type. That implies that for $k\geq 7$, $F_k(q)$ is neither itself Bloch--P\'olya, nor can be corrected to be one by an addition of a polynomial.
\end{enumerate}
\end{proof}

\section{Further Observations}\label{Beyond}

Another topic to address is the classification of $(q;q)_m$ polynomials with coefficients from the set $\{-h,-h+1,\dots ,h-1, h\}$, for any positive integer $h$. Let $S_h$ be the set of all the $m$ values such that the coefficients of $(q;q)_m$ lie in between $-h$ and $h$, where at least one coefficient has the absolute value $h$. We already proved that $S_1 = \{0,1,2,3,5\}$ using \eqref{q_sub_m}. This argument can be repeated to find all $S_h$ for $h\geq 2$. As an example, from \eqref{coeffSizes}, it is easy to see that for all $m> 116$ \[3\leq \llfloor q^{2m+116} \rrfloor(q;q)_{m-1} \leq 7.\] Therefore, $m=116$ is the cut-off point for $S_2$ and one only needs to check  $m \leq 116$ manually to find all $m$ values in $S_2$. The general formula for the cut-off points for $S_h$ are \[p_1(2h+5)-1 = (h+2)(6h+17),\] where $p_1(n)$ is defined as in \eqref{penta_families}.

We display more sets, that are confirmed, and their related cut-off points in Table~\ref{Table_S_h}.

\begin{table}[h]\caption{List of $S_h$ for $h=1\dots 40$ with the cut-off values of $m$.}
\label{Table_S_h}
\scalebox{0.8}{\begin{tabular}{ccc||ccc||ccc||ccc}
$h$ & $S_h$ & Cut-off & $h$ & $S_h$ & Cut-off & $h$ & $S_h$ & Cut-off& $h$ & $S_h$ & Cut-off\\\hline
1&$\{0,1,2,3,5\}$& 69& 11 & $\{23\}$ & 1079 & 21&$\{27\}$&3289 & 31&$\emptyset$&6699\\
2&$\{4,6,7,8,9,11\}$&116& 12 & $\emptyset$&1246&22&$\emptyset$&3576&32&$\emptyset$&7106\\
3&$\{10,13,14\}$&175& 13&$\emptyset$&1425&23&$\emptyset$&3875&33&$\emptyset$&7525\\
4&$\{12,15\}$&246&14 &$\emptyset$& 1616&24&$\emptyset$&4186&34&$\{30\}$&7956\\
5&$\{17\}$ &329&15&$\emptyset$&1819&25&$\emptyset$&4509&35&$\emptyset$&8399\\
6&$\{16,18\}$&424&16&$\{24,25\}$&2034&26&$\emptyset$&4844&36&$\emptyset$&8854\\
7&$\{19\}$&531& 17&$\emptyset$&2261&27&$\emptyset$&5191&37&$\emptyset$&9321\\
8&$\{20,21\}$&650& 18&$\emptyset$&2500&28&$\{28\}$&5550&38&$\emptyset$&9800\\
9&$\emptyset$&781& 19&$\{26\}$&2751&29&$\{29\}$&5921&39&$\emptyset$&10291\\
10&$\{22\}$&924&20&$\emptyset$&3014&30&$\emptyset$&6304&40&$\emptyset$&10794\\
\end{tabular}}
\end{table}

The data in Table~\ref{Table_S_h} is consistent with the following.

\begin{conjecture} \[\text{Either }S_h = \emptyset\text{,  or  }S_h = \{i(h)\},\]
for $h>16$, where $i(h)$ is a positive integer, and \[i(h_1)>i(h_2)\text{  when  }h_1>h_2>16.\]
Moreover, for $h>5$, the set \[S_1\cup S_2\cup\dots\cup S_h=\{0,1,2,\dots ,M(h)\}\] consists of all consecutive integers from 0 up to some positive $M(h)$.
\end{conjecture}

Similarly, one can also define the set $\hat{S}_h$ for the series $F_k(q)$. Let $\hat{S}_h$ be the set of positive integers $k$ such that $F_k(q)$ has its coefficients from the set $\{-h,\dots, h\}$, where at least one coefficient has the absolute value $h$. Theorem~\ref{main_F_THM} shows that $1,2 \in \hat{S}_1$ and $3,4,6\in\hat{S}_2$. Moreover, similar to the proof of Theorem~\ref{main_F_THM}, using Lemma~\ref{penta_gap}, we can easily find the cut-off points, making sure that the pentagonal numbers are farther apart from the degree of $(q;q)_{k-1}$. Using this cut-off point, one can identify which $\hat{S}_h$ set $F_k(q)$ lies in by looking at the initial coefficients of $F_k(q)$ and the coefficients of $(q;q)_{k-1}$. For example, recall \eqref{tail3}, \[F_5(q) = P(q) + (q;q)_4(q^{161}+ q^{172}  -q^{195} -q^{207}\dots),\] where \[P(q) = 1 + q^5 + \dots -2q^{21} + \dots+ 3q^{30} + \dots + q^{150}.\] 
The polynomial $P(q)$ have all of its coefficients between $-2$ and 3. There is more than $10$ difference between all the exponents of $q$ with non-zero coefficients in the Bloch--P\'olya type series $q^{161}+ q^{172}  -q^{195} -q^{207}\dots$. The polynomial $(q;q)_4$ has degree $10$ and its largest absolute coefficient is 2. Hence, $(q;q)_4(q^{161}+ q^{172}  -q^{195} -q^{207}\dots)$ is a series with all its coefficients from the set $\{-2,\dots ,2\}$. Comparing $P(q)$ polynomial and the tail end of $F_5(q)$ we deduce that the maximum absolute coefficient of $F_5(q)$ is 3. Therefore, $5\in\hat{S}_3$. In general, it is sufficient to check the coefficients of $F_k(q)$ until the exponent \[ p_1( k(k-1)/2 + 1 )-k = \frac{(k-1)(3k^3-3k^2+10k-8)}{8}\] to classify its respective $\hat{S}_h$ set, where $p_1(n)$ is defined as in \eqref{penta_families}. This bound is used in comparison of the coefficients of $(q;q)_{k-1}$, which appears repeatedly as shifted copies in the tail end of $F_k(q)$, with the initial coefficients of $F_k(q)$.

We give a list of confirmed $\hat{S}_h$ sets in Table~\ref{Table_F_k}.

\begin{table}[h]\caption{List of $\hat{S}_h$ for $h=1\dots 42$.}
\label{Table_F_k}
\scalebox{0.8}{\begin{tabular}{cc||cc||cc||cc||cc||cc||cc}
$h$ & $\hat{S}_h$ 	&  $h$ 	& $\hat{S}_h$ 	&  $h$ 	& $\hat{S}_h$ &  $h$ 	& $\hat{S}_h$&  $h$ 	& $\hat{S}_h$&  $h$ 	& $\hat{S}_h$&  $h$ 	& $\hat{S}_h$\\\hline
1	&$\{1,2\}$		& 7	&$\{11,14\}$		&13	&$\emptyset$& 19	&$\emptyset$&25		&$\emptyset$&31		&$\emptyset$&37		&$\{25\}$\\
2	&$\{3,4,6\}$	& 8	&$\{13,15\}$		&14 	&$\{18\}$& 20	&$\emptyset$&26		&$\emptyset$&32		&$\emptyset$&38		&$\emptyset$\\
3	&$\{5,8\}$		& 9	&$\emptyset$		&15	&$\{19\}$&21		&$\emptyset$&27		&$\emptyset$&33		&$\emptyset$&39		&$\emptyset$\\
4	&$\{7,9\}$		& 10	&$\emptyset$	&16	&$\emptyset$&22		&$\emptyset$&28		&$\emptyset$&34		&$\emptyset$&40		&$\emptyset$\\
5	&$\emptyset$ 	& 11 	& $\{16\}$		&17	&$\{20\}$&23		&$\emptyset$&29		&$\{23\}$&35		&$\emptyset$&41	&$\emptyset$\\
6	&$\{10,12\}$	& 12 	& $\{17\}$		&18	&$\{21\}$	&24		&$\{22\}$	&30		&$\{24\}$&36		&$\emptyset$&42	&$\{26\}$
\end{tabular}}
\end{table}
\section{Acknowledgement}

We would like to thank George Andrews, Frank Garvan, Michael Mossinghoff, Larry Rolen, William Severa and the anonymous referee for their kind interest and helpful comments. 

Research of the first author is partly supported by the Simons Foundation, Award ID: 308929.

\end{document}